\documentclass[11pt]{amsart}

\usepackage[T1]{fontenc}
\usepackage[utf8]{inputenc}

\usepackage{amsmath,amssymb,amsfonts,amsthm,mathtools}
\usepackage{mathrsfs}
\usepackage{enumitem}
\usepackage{tikz}
\usepackage[colorlinks=true,
            linkcolor=blue,
            citecolor=blue,
            urlcolor=blue]{hyperref}
\usepackage[nameinlink,noabbrev]{cleveref}
\usepackage{cite}

\numberwithin{equation}{section}

\newtheorem{theorem}{Theorem}[section]

\newtheorem{proposition}[theorem]{Proposition}

\theoremstyle{definition}
\newtheorem{definition}[theorem]{Definition}

\title{On the Transitive Binary $G$-Spaces}

\author{Pavel S. Gevorgyan}
\address{%
\scalebox{0.94}[1]{Department of Mathematical Analysis named after Academician P.S. Novikov}\\
Moscow Pedagogical State University, Moscow, Russia}
\email{pgev@yandex.ru}


\begin{document}

\begin{abstract}
Distributive subsets of the group of all invertible continuous binary operations on a
 topological space are considered, and it is proved that the subgroups generated by them are also
 distributive. A criterion for the distributivity of a binary action of a topological group $G$ on a space $X$
 is obtained. The concept of transitive binary $G$-space is introduced, and a classification of transitive
 distributive binary $G$-spaces is given in the case of a compact group $G$.
\end{abstract}

\maketitle

\section{Introduction}

Representations of groups by invertible binary operations of some set play an important role in various
questions of algebra and other branches of mathematics. In terms of binary representations of
groups, one of the old problems of algebra was solved: the problem of describing those
groups that can be multiplicative groups of fields \cite{movsisyan-1}.

Studying binary representations of a topological group $G$, or binary $G$-spaces, dates back to
work \cite{Gev}. In a more general sense, a representation of a topological group $G$ should be
understood as a homomorphism of the group $G$ into the group of all invertible continuous binary
operations of a space $X$. The category of binary $G$-spaces and biequivariant maps is an extension
of the category of $G$-spaces and equivariant maps. The class of distributive binary $G$-spaces plays
a special role; important notions and results of equivariant topology extend to this class
\cite{Gev2,Gev1,Gev-Iliadis,Gev-Naz,Gev-3,Gev4,Gev-Quitzeh}.

In the classification of $G$-spaces, as is well known, the central role is played by the category of
transitive $G$-spaces, or the category of $G$-orbits \cite{Palais,Bredon,Dieck}. In the present paper
we define the notion of a transitive binary $G$-space and provide a complete classification of these
spaces in the class of distributive binary $G$-spaces with a compact group $G$ (Definition~\ref{def-1}
and Theorem~\ref{prop-4}). For this purpose, we study the distributive subsets of the group of all
invertible continuous binary operations of a topological space $X$ (Propositions~\ref{cor-obratim},
\ref{cor-obratim-2} and Theorem~\ref{th-10}), prove a criterion for the distributivity of a binary action
of a group $G$ on $X$ (Theorem~\ref{th-11}), and establish several properties of stationary subgroups
$G_{(x,y)}$ of the pair $(x,y)$, $x,y\in X$ (Propositions~\ref{prop-10}--\ref{prop-11}).

\section{Preliminary Information}

Let $X$ be a topological space. By $C_2(X)$ we denote the space of all continuous binary operations
$f:X^2\to X$ with the compact-open topology. The space $C_2(X)$ with the operation
\begin{equation*}\label{eq-operation}
(f\varphi)(x,y)=f(x,\varphi(x,y)),
\end{equation*}
where $x,y\in X$, is a topological monoid with unity $e(x,y)=y$.

By $H_2(X)$ we denote the group of all invertible binary operations of the monoid $C_2(X)$.

Let $G$ be a topological group. A continuous map
$\alpha:G\times X^2\to X$ is called a \textit{binary action} of the group $G$
on a space $X$ if, for arbitrary $g,h\in G$ and $x,y\in X$, the following
equalities hold:
\begin{equation*}
gh(x,y)=g(x,h(x,y)),
\end{equation*}
\begin{equation*}
e(x,y)=y,
\end{equation*}
where $e$ is the identity element of the group $G$, and
$g(x,y)=\alpha(g,x,y)$. A space $X$ with a fixed binary action $\alpha$ of
the group $G$, or a triple $(G,X,\alpha)$, is called a \textit{binary
$G$-space}.

Any topological group $G$ acts binarily on itself by means of the
\emph{conjugate left translation}
\begin{equation}\label{eq-H on G}
g(x,y)=xgx^{-1}y,
\end{equation}
where $g,x,y\in G$.

An arbitrary element $g\in G$ generates a continuous binary operation
$\alpha_g:X^2\to X$ by the formula
\begin{equation*}
\alpha_g(x,y)=\alpha(g,x,y).
\end{equation*}
The map $g\to \alpha_g$ defines a homomorphism
\[
\theta:G\to H_2(X)
\]
of the group $G$ into the group $H_2(X)$ of all invertible continuous binary
operations of the space $X$. The kernel of the homomorphism $\theta$ is called
the \emph{kernel} of the binary action $\alpha$. Thus,
\[
\ker\alpha
=
\{g\in G;\quad g(x,y)=y,\ \text{for all } x,y\in X\}.
\]
It is not hard to see that $\ker\alpha$ is a closed normal subgroup of the
group $G$.

A binary action $\alpha$ is called \emph{effective} if $\ker\alpha=e$, that is,
the map $\theta$ is injective.

Let $(G,X,\alpha)$ and $(G,Y,\beta)$ be binary $G$-spaces. A continuous map
$f:X\to Y$ is called \emph{biequivariant} if the condition
\[
f(\alpha(g,x,y))=\beta(g,f(x),f(y))
\]
is fulfilled, or equivalently,
\[
f(g(x,y))=g(f(x),f(y))
\]
for all $g\in G$ and $x,y\in X$.

A biequivariant map $f:X\to Y$ which is also a homeomorphism is called an
\textit{equivalence} of binary $G$-spaces, or a \textit{biequimorphism}.

All binary $G$-spaces and biequivariant maps form a category which is a
natural extension of the category of $G$-spaces and equivariant maps.

A binary $G$-space $(G,X,\alpha)$ is called \textit{distributive} if, for
arbitrary $x,y,z\in X$ and $g,h\in G$, the following equality holds:
\begin{equation}\label{eq1-1}
g(x,h(y,z))=h(g(x,y),g(x,z)).
\end{equation}
In this case $\alpha$ is called a \emph{distributive} binary action of the
group $G$ on $X$.

The binary action \eqref{eq-H on G} of the group $G$ on itself by means of the
conjugate left translation is distributive, whereas the binary action of the group
$G$ on itself given by the formula
$g(x,y)=x^{-1}gxy$, $g,x,y\in G$, does not possess this property.

\section{Distributive Subsets of the Group of Invertible Binary Operations
on a Topological Space}

Let $g,h\in H_2(X)$ be invertible binary operations of the space $X$. We call
an ordered pair $(g,h)$ a \emph{distributive pair} if it satisfies the identity
\eqref{eq1-1}. A subset $D\subset H_2(X)$ is called \emph{distributive} if all
its elements are pairwise distributive.

For an arbitrary $x\in X$, we define a map $g_x:X\to X$ by the formula
$g_x(y)=g(x,y)$. Note that the following equality holds:
\begin{equation}\label{eq-g_xh_x}
(gh)_x=g_xh_x
\end{equation}
for arbitrary $g,h\in H_2(X)$ and $x\in X$.

\begin{proposition}\label{th-criteria}
Let $g,h\in H_2(X)$ be invertible binary operations of the space $X$. Then the
following statements are equivalent:

1. The pair $(g,h)$ is a distributive pair.

2. The map $g_x:X\to X$ satisfies the equality
\[
g_x(h(y,z))=h(g_x(y),g_y(z))
\]
for all $x,y,z\in X$.

3. The equality
\[
g_xh_y=h_{g_x(y)}g_x
\]
holds for all $x,y\in X$.
\end{proposition}

We omit the proof of this proposition, since it is rather simple.

\begin{proposition}\label{cor-obratim}
Let $g,h\in H_2(X)$ be invertible binary operations of the space $X$. If
$(g,h)$ is a distributive pair, then $(g,h^{-1})$, $(g^{-1},h)$, and
$(g^{-1},h^{-1})$ are also distributive pairs.
\end{proposition}

\begin{proof}
Let $(g,h)$ be a distributive pair. By Proposition~\ref{th-criteria}, we have
$g_xh_y=h_{g_x(y)}g_x$ for arbitrary $x,y\in X$. This equality implies that
$h_y^{-1}=g_x^{-1}h_{g_x(y)}^{-1}g_x$, or
$g_xh_y^{-1}=h_{g_x(y)}^{-1}g_x$. By the same Proposition~\ref{th-criteria},
this means that $(g,h^{-1})$ is a distributive pair.

Now observe that
\begin{multline*}
g^{-1}(x,h(y,z))
=
g^{-1}(x,h(g(x,g^{-1}(x,y)),g(x,g^{-1}(x,z))))\\
=
g^{-1}(x,g(x,h(g^{-1}(x,y),g^{-1}(x,z))))
=
h(g^{-1}(x,y),g^{-1}(x,z)).
\end{multline*}
This means that $(g^{-1},h)$ is a distributive pair.

The distributivity of the pair $(g^{-1},h^{-1})$ now follows from the
distributivity of the pair $(g,h^{-1})$ or of the pair $(g^{-1},h)$.
\end{proof}

\begin{proposition}\label{cor-obratim-2}
Let $g,h,k\in H_2(X)$ be invertible binary operations of the space $X$. If
$(g,h)$ and $(g,k)$ are distributive pairs, then $(g,hk)$ is also a
distributive pair.
\end{proposition}

\begin{proof}
By Proposition~\ref{th-criteria}, for arbitrary $x,y\in X$ the following
equalities hold:
\[
g_xh_y=h_{g_x(y)}g_x
\quad\text{and}\quad
g_xk_y=k_{g_x(y)}g_x.
\]
It is not hard to see that then
$g_xh_yk_y=h_{g_x(y)}k_{g_x(y)}g_x$, or, taking into account~\eqref{eq-g_xh_x},
\[
g_x(hk)_y=(hk)_{g_x(y)}g_x.
\]
Therefore, by Proposition~\ref{th-criteria}, $(g,hk)$ is a distributive pair.
\end{proof}

The following two theorems immediately follow from Propositions \ref{th-criteria},
\ref{cor-obratim}, and \ref{cor-obratim-2}.

\begin{theorem}\label{th-10}
Let $D\subset H_2(X)$ be a distributive subset. Then the subgroup of the group
$H_2(X)$ generated by the set $D$ is a distributive subgroup.
\end{theorem}

\begin{theorem}\label{th-11}
Let $X$ be a binary $G$-space. Then the following statements are equivalent:

1. The action of the group $G$ on $X$ has the distributivity property.

2. The map $g_x:X\to X$ is a biequivariant homeomorphism for all $g\in G$
and $x\in X$.

3. For all $g,h\in G$ and $x,y\in X$, the equality
\[
g_xh_y=h_{g_x(y)}g_x
\]
holds.
\end{theorem}

\section{Stationary Subgroups and Transitive Binary \texorpdfstring{$G$}{G}-Spaces}

Let $X$ be a binary $G$-space, and let $x,y\in X$ be arbitrary points. The set
\[
G_{(x,y)}=\{g\in G,\ g(x,y)=y\}
\]
is a closed subgroup of the group $G$. This subgroup is called the
\emph{stationary subgroup} of the pair $(x,y)$. The subgroup $G_{(x,x)}$ is
called the \emph{stationary subgroup of the point} $x$.

\begin{proposition}\label{prop-10}
Let $f:X\to Y$ be a biequivariant map between binary $G$-spaces $X$ and
$Y$. Then
\begin{equation*}
G_{(x,y)}\subset G_{(f(x),f(y))}
\end{equation*}
for any $x,y\in X$.
\end{proposition}

This proposition can be proved elementary.

\begin{proposition}\label{prop-2}
Let $X$ be an arbitrary binary $G$-space, and let $x,y\in X$ and $g\in G$ be
arbitrary. Then the equality
\begin{equation}\label{G{(x, g(x,x'))}}
G_{(x,g(x,y))}=gG_{(x,y)}g^{-1}
\end{equation}
holds.
\end{proposition}

\begin{proof}
Let $g_0\in G_{(x,y)}$ be an arbitrary element. Then, for any $g\in G$, we have
\[
gg_0g^{-1}(x,g(x,y))=gg_0(x,y)=g(x,g_0(x,y))=g(x,y).
\]
Hence,
\[
gG_{(x,y)}g^{-1}\subset G_{(x,g(x,y))}.
\]

In particular, the last inclusion implies that
\[
g^{-1}G_{(x,g(x,y))}g
\subset
G_{(x,g^{-1}(x,g(x,y)))}=G_{(x,y)}.
\]
From this we obtain the inverse inclusion
\[
G_{(x,g(x,y))}\subset gG_{(x,y)}g^{-1}.
\]
\end{proof}

\begin{proposition}\label{prop-1}
Let $X$ be a distributive binary $G$-space, and let $x\in X$ be an arbitrary
point. Then the stationary subgroup $G_{(x,x)}$ is a closed normal subgroup of
the group $G$.
\end{proposition}

\begin{proof}
Let $g_0\in G_{(x,x)}$ and $g\in G$ be arbitrary elements. By distributivity,
we have
\begin{multline*}
gg_0g^{-1}(x,x)
=
g(x,g_0(x,g^{-1}(x,x)))\\
=
g(x,g^{-1}(g_0(x,x),g_0(x,x)))
=
g(x,g^{-1}(x,x))
=
x,
\end{multline*}
that is, $gg_0g^{-1}\in G_{(x,x)}$. Hence,
$gG_{(x,x)}g^{-1}=G_{(x,x)}$, which means that $G_{(x,x)}$ is a normal
subgroup.
\end{proof}

\begin{proposition}\label{prop-11}
Let $X$ be a distributive binary $G$-space. Then, for arbitrary
$x,y,z\in X$ and $g\in G$, the following equalities hold:
\begin{equation}\label{G(h(x,y)}
G_{(g(x,y),g(x,z))}=G_{(y,z)},
\end{equation}
\begin{equation}\label{eq-Gxgx}
G_{(x,g(x,x))}=G_{(x,x)},
\end{equation}
\begin{equation}\label{eq-3}
G_{(g(x,x),x)}=G_{(x,x)}.
\end{equation}
\end{proposition}

\begin{proof}
Let us prove equality \eqref{G(h(x,y)}. Suppose that $g_0\in G_{(y,z)}$. Then
\begin{equation*}\label{eq-distr}
g_0(g(x,y),g(x,z))=g(x,g_0(y,z))=g(x,z).
\end{equation*}
This means that $g_0\in G_{(g(x,y),g(x,z))}$.

Conversely, let $g_0\in G_{(g(x,y),g(x,z))}$, that is,
$g_0(g(x,y),g(x,z))=g(x,z)$. Then
\begin{equation*}\label{eq-distr-1}
g(x,g_0(y,z))=g_0(g(x,y),g(x,z))=g(x,z).
\end{equation*}
Hence, $g_0(y,z)=z$, that is, $g_0\in G_{(y,z)}$.

Equality \eqref{eq-Gxgx} immediately follows from formula
\eqref{G{(x, g(x,x'))}} and Proposition~\ref{prop-1}:
\[
G_{(x,g(x,x))}=gG_{(x,x)}g^{-1}=G_{(x,x)}.
\]

Let us prove \eqref{eq-3}. By \eqref{G(h(x,y)} and \eqref{eq-Gxgx}, we have
\[
G_{(g(x,x),x)}
=
G_{(g(x,x),g(x,g^{-1}(x,x)))}
=
G_{(x,g^{-1}(x,x))}
=
G_{(x,x)}.
\]
\end{proof}

\begin{definition}\label{def-1}
A binary action of a group $G$ on $X$ is called \emph{transitive} if
$G(x,x)=X$ for all $x\in X$. In this case $X$ is called a \emph{transitive}
binary $G$-space.
\end{definition}

\begin{proposition}\label{prop-3}
Let $H$ be a closed normal subgroup of a group $G$. The map
$\alpha:G\times G/H\times G/H\to G/H$, defined by the formula
\begin{equation}\label{eq-g(kH,lH)}
g(kH,lH)=(kgk^{-1}l)H,
\end{equation}
where $g(kH,lH)=\alpha(g,kH,lH)$, is a distributive binary action of the group
$G$ on $G/H$.
\end{proposition}

\begin{proof}
First we prove that formula \eqref{eq-g(kH,lH)} defines a well-defined
map. Suppose that $kH=mH$ and $lH=nH$, where $k,l,m,n\in G$. We show that
$g(kH,lH)=g(mH,nH)$, that is,
\[
(kgk^{-1}l)H=(mgm^{-1}n)H,
\]
or, equivalently,
\[
(kgk^{-1}l)H=(mgm^{-1}l)H.
\]
For this purpose, it is enough to prove that
\[
(mgm^{-1}l)^{-1}(kgk^{-1}l)\in H.
\]
Indeed, let $m=kh_1$ for some $h_1\in H$. Then
\begin{multline*}
(mgm^{-1}l)^{-1}(kgk^{-1}l)
=
l^{-1}mg^{-1}m^{-1}kgk^{-1}l
=
l^{-1}kh_1g^{-1}h_1^{-1}k^{-1}kgk^{-1}l\\
=
l^{-1}k(h_1g^{-1}h_1^{-1}g)k^{-1}l
=
l^{-1}kh_2k^{-1}l
\in H,
\end{multline*}
because the element $h_2=h_1g^{-1}h_1^{-1}g$ belongs to $H$.

It is easy to check that formula \eqref{eq-g(kH,lH)} defines a binary action
of the group $G$ on $G/H$. Let us prove that this action is distributive, that
is,
\[
g(kH,g'(lH,mH))=g'(g(kH,lH),g(kH,mH)).
\]
Indeed,
\[
g(kH,g'(lH,mH))
=
g(kH,lg'l^{-1}mH)
=
kgk^{-1}lg'l^{-1}mH.
\]
On the other hand,
\begin{multline*}
g'(g(kH,lH),g(kH,mH))
=
g'(kgk^{-1}lH,kgk^{-1}mH)\\
=
kgk^{-1}lg'(kgk^{-1}l)^{-1}kgk^{-1}mH
=
kgk^{-1}lg'l^{-1}mH.
\end{multline*}
\end{proof}

The space $G/H$ with the binary action \eqref{eq-g(kH,lH)} of the group $G$ is
an example of a transitive distributive binary $G$-space.

\begin{theorem}\label{prop-4}
Any transitive distributive binary $G$-space $X$, where $G$ is a compact group,
is biequivariantly homeomorphic to the space of cosets $G/H$ by some closed
normal subgroup $H$ of the group $G$, with the binary action~\eqref{eq-g(kH,lH)}.
\end{theorem}

\begin{proof}
Let $X$ be an arbitrary transitive distributive binary $G$-space, and let
$x_0\in X$ be a fixed point. By Proposition~\ref{prop-1}, the stationary
subgroup $G_{(x_0,x_0)}=H$ is a closed normal subgroup of the group $G$.

Consider the space of cosets $G/H$ with the binary action
\eqref{eq-g(kH,lH)}. Define the map $f:G/H\to X$ by the formula
\begin{equation}\label{eq-f(gH)}
f(gH)=g(x_0,x_0)
\end{equation}
for an arbitrary $gH\in G/H$. It is not hard to prove that $f$ is continuous
and bijective. Since a continuous bijective map from a compact space onto a
Hausdorff space is a homeomorphism, $f$ is a homeomorphism.

It remains to prove that $f$ is a biequivariant map, that is, the equality
\[
f(g(kH,lH))=g(f(kH),f(lH))
\]
holds for all $g,k,l\in G$. Taking into account \eqref{eq-g(kH,lH)},
\eqref{eq-f(gH)}, and the distributivity of the action of the group $G$ on
$X$, we obtain
\begin{align*}
g(f(kH),f(lH))
&= g\bigl(k(x_0,x_0),l(x_0,x_0)\bigr) \\
&= g\bigl(k(x_0,x_0),k(x_0,k^{-1}l(x_0,x_0))\bigr) \\
&= k\bigl(x_0,g(x_0,k^{-1}l(x_0,x_0))\bigr) \\
&= k\bigl(x_0,gk^{-1}l(x_0,x_0)\bigr) \\
&= kgk^{-1}l(x_0,x_0) \\
&= f\bigl((kgk^{-1}l)H\bigr) \\
&= f\bigl(g(kH,lH)\bigr).
\end{align*}
\end{proof}

The following theorem immediately follows from Proposition~\ref{prop-3} and
Theorem~\ref{prop-4}.

\begin{theorem}
Let a binary action of a compact group $G$ on a space $X$ be effective,
distributive, and transitive. Then $X$ is biequivariantly homeomorphic to the
topological group $G$ with the binary action by means of the conjugate left
translation.
\end{theorem}

As shown by the next proposition, a group $G$ with the binary action by means
of the conjugate left translation is the only distributive transitive binary
$G$-space with an effective binary action.

\begin{proposition}
Let $\alpha:G\times X^2\to X$ be a transitive distributive binary
$G$-space, and let $x\in X$ be an arbitrary point. Then the kernel of the
binary action $\alpha$ coincides with the stationary subgroup of the point
$x$:
\[
\ker\alpha=G_{(x,x)}.
\]
\end{proposition}

\begin{proof}
It suffices to prove that the equality $G_{(y,z)}=G_{(x,x)}$ holds for
arbitrary points $y,z\in X$. By the transitivity of the binary action, there
exist elements $g,g'\in G$ such that $g(x,x)=y$ and $g'(x,x)=z$. Now, taking
\eqref{G(h(x,y)} and \eqref{eq-Gxgx} into account, we obtain
\[
G_{(y,z)}
=
G_{(g(x,x),g'(x,x))}
=
G_{(g(x,x),g(x,g^{-1}g'(x,x)))}
=
G_{(x,g^{-1}g'(x,x))}
=
G_{(x,x)}.
\]
\end{proof}


\begin{thebibliography}{99}

\bibitem{movsisyan-1}
Yu.~M. Movsisyan,
``The multiplicative group of a field and hyperidentities,''
\emph{Math. USSR-Izv.} \textbf{35} (1990), 377--391.

\bibitem{Gev}
P.~S. Gevorkyan,
``On binary $G$-spaces,''
\emph{Math. Notes} \textbf{96} (2014), 600--602.

\bibitem{Gev2}
P.~S. Gevorgyan,
``Groups of binary operations and binary $G$-spaces,''
\emph{Topology Appl.} \textbf{201} (2016), 18--28.

\bibitem{Gev1}
P.~S. Gevorgyan,
``Groups of invertible binary operations of a topological space,''
\emph{J. Contemp. Math. Anal.} \textbf{53} (2018), 16--20.

\bibitem{Gev-Iliadis}
P.~S. Gevorgyan and S.~D. Iliadis,
``Groups of generalized isotopies and generalized $G$-spaces,''
\emph{Mat. Vesnik} \textbf{70} (2018), no.~2, 110--119.

\bibitem{Gev-Naz}
P.~S. Gevorgyan and A.~A. Nazaryan,
``On orbits and bi-invariant subsets of binary $G$-spaces,''
\emph{Math. Notes} \textbf{109} (2021), 38--45.

\bibitem{Gev-3}
P.~S. Gevorgyan,
``On orbit spaces of distributive binary $G$-spaces,''
\emph{Math. Notes} \textbf{112} (2022), 177--182.

\bibitem{Gev4}
P.~S. Gevorgyan,
``Bi-equivariant fibrations,''
\emph{Topology Appl.} \textbf{329} (2023), 108361.

\bibitem{Gev-Quitzeh}
P.~S. Gevorgyan and Q. Morales Melendez,
``Universal space for binary $G$-spaces,''
\emph{Topology Appl.} \textbf{329} (2023), 108370.

\bibitem{Palais}
R.~S. Palais,
\emph{The Classification of $G$-Spaces},
Memoirs of the American Mathematical Society, vol.~36,
American Mathematical Society, Providence, RI, 1960.

\bibitem{Bredon}
G. Bredon,
\emph{Introduction to Compact Transformation Groups},
Pure and Applied Mathematics, vol.~46,
Academic Press, London, 1972.

\bibitem{Dieck}
T. tom Dieck,
\emph{Transformation Groups},
De Gruyter Studies in Mathematics, vol.~8,
De Gruyter, Berlin, 1987.

\end{thebibliography}
\end{document}